\documentclass[12pt]{amsart}
\usepackage{amsmath}
\usepackage{amssymb}
\usepackage{amsfonts}
\usepackage{amsthm}
\usepackage{xcolor}
\usepackage{enumerate}

\newtheorem{theo}{Theorem}[section]
\newtheorem*{theo*}{Theorem}

\newtheorem{lemm}[theo]{Lemma}

\newcommand{\N}{\mathbb{N}}
\newcommand{\Z}{\mathbb{Z}}

\newcommand{\C}{\mathbb{C}}

\title{A Menon-type Identity derived using Cohen-Ramanujan sum}

\begin{document}
\author[A Chandran]{Arya Chandran}
\address{Department of Mathematics, University College, Thiruvananthapuram, Kerala - 695034, India}
\email{aryavinayachandran@gmail.com}
\author[K V Namboothiri]{K Vishnu Namboothiri}
\address{Department of Mathematics, Government College, Chittur, Palakkad, Kerala - 678104, INDIA\\Department of Collegiate Education, Government of Kerala, India}
\email{kvnamboothiri@gmail.com}

 \begin{abstract}
  Menon's identity is a classical identity involving gcd sums and the Euler totient function $\phi$. We derived the Menon-type  identity $\sum\limits_{\substack{m=1\\(m.n^s)_s=1}}^{n^s} (m-1,n^s)_s=\Phi_s(n^s)\tau_s(n^s)$  in [\textit{Czechoslovak Math. J., 72(1):165–176 (2022)}] where $\Phi_s$ denotes the Klee's function and $(a,b)_s$ denotes a a generalization of the gcd function. Here we give an alternate method to derive this identity using the concept of Cohen-Ramanujan sum.
  \end{abstract}

 \maketitle

 \section{Introduction}
 The classical Menon's identity which originally appeared in \cite{menon1965sum} is a gcd sum turning out to be equal to a product of the Euler totient function $\phi$ and the number of divisors function $\tau$. If $(m,n)$ denotes the gcd of $m$ and $n$, the identity is precisely the following:
\begin{align}\label{menons-identity}
\sum\limits_{\substack{m=1\\(m.n)=1}}^n (m-1,n)=\phi(n)\tau(n).
\end{align}
It has been generalized and extended by many authors. For example, see the papers \cite{li2017menon} \cite{sury2009some}, \cite{toth2018menon} and \cite{zhao2017another}.
For positive integer $s$, integers $a,b$, not both zero, the largest $l^s$ (where $l\in \N$) dividing both $a$ and $b$ denoted by  $(a,b)_{s}$ is called the generalized gcd of $a$ and $b$ (following E. Cohen \cite{cohen1956some}). When $s=1$, this will be equal to the usual gcd of $a$ and $b$.

For positive integers $s$ and $n$, the Klee's function  $\Phi_{s}(n)$ is defined to give the cardinality of the  set $\{m\in\N : 1\leq m\leq n, (m,n)_s=1\}$.

Li and Kim  gave a generalization of the identity (\ref{menons-identity}) in  \cite{li2017menon} generalizing  which further we we proved the following result.

\begin{theo}\cite[Theorem 3.2]{chandran2022menon}
Let $m_1,m_2,\cdots,m_k, b_1,b_2,\cdots, b_r,n,s \in \N$ and $a_1,a_2,\cdots, a_k \in \Z$ be such that $(a_i,n^s)_s=1$, $i = 1,2,\cdots,k$. Then
\begin{align}
\sum\limits_{\substack{1\leq m_1, m_2,\ldots,m_k\leq n^s\\(m_1,n^s)_s=1\\(m_2,n^s)_s=1\\\cdots \\(m_k,n^s)_s=1\\1\leq b_1, b_2,\ldots,b_r\leq n^s}}(m_1-a_1, m_2-a_2,&\ldots,m_k-a_k,b_1,b_2,\cdots,b_r,n^s)_s\\
&= \Phi_s(n^s)^k\sum\limits_{\substack{d^s|n^s}}\frac{(d^s)^r}{\Phi_s(\frac{n^s}{d^s})^{k-1}}.\nonumber
\end{align}
\end{theo}

As a corollary to this, we derived the Menon-type identity  {{\cite[Corollary 3.1]{chandran2022menon}}}: \begin{align}\label{this-paper-idn}
                                                                                       \sum\limits_{\substack{m=1\\(m.n^s)_s=1}}^{n^s} (m-1,n^s)_s=\Phi_s(n^s)\tau_s(n^s).
                                                                                      \end{align}

   The Ramanujan sum denoted by $c_r(n)$  \cite{ramanujan1918certain}  is defined to be the sum of certain powers of a primitive $r$th root of unity. That is,
\begin{align*}
c_r(n)&=\sum\limits_{\substack{{m=1}\\(m,r)=1}}^{r}e^{\frac{2 \pi imn}{r}}
\end{align*}
 where $r\in\mathbb{N}$ and $n \in \mathbb{Z}$.

    Ramanujan sum has been generalized in many ways. E. Cohen \cite{cohen1949extension} defined
\begin{align}\label{Coh1}
c_r^s(n)&=\sum\limits_{\substack{{(h,r^s)_s=1}\\h=1}}^{r^s}e^{\frac{2\pi i n h}{r^s}}.
\end{align}
When $s=1$, this reduces to the Ramanujan sum. This sum will be henceforth  called as the Cohen-Ramanujan sum. Some other important properties of this sum is given in the next section.

Our proof of the above theorem \cite{chandran2022menon} used elementary number theory techniques. The original proof of the Menon's identity was also given using elementary number theory technques. In \cite{toth2021proofs}, L. T\'{o}th gave a proof of the classical Menon's identity (\ref{menons-identity}) using the concept of Ramanujan sum. This was surprising as most of the other related identities were derived using either elementary number theory techniques, group theory or some sort of character theory. In this paper we derive the Menon-type identity (\ref{this-paper-idn}) using the properties of the generalized gcd function and the Cohen-Ramanujan sum.

\section{Notations and basic results}
 A postive integer $a$ is said to be $s-$power free or $s-$free if no $l^s$ where $l\in \N$  divides $a$.

 The Euler totient function $\phi$ has several generalizations of which three are relevant to us. They are the Jordan totient function, Cohen totient function and the Klee's function. The Jordan totient function $J_s(n)$   \cite[pp 95-97]{jordan1870traite} defined for positive integers $s$ and $n$ gives the number of ordered sets of $s$ elements from a complete residue system (mod $n$) such that the greatest common divisor of each set is prime to $n$. The Cohen totient  function $\phi_s$  \cite{cohen1956some} is defined as follows. If $(a,b)_s=1$, then $a$ and $b$ are said to be relatively $s-$prime. The subset $N$ of a complete residue system $M$ (mod $n^s$) consisting of all elements of $M$ that are relatively $s-$prime to $n^s$ is called an $s-$reduced residue system (mod $n$). Cohen totient function $\phi_s(n)$ gives the number of elements of an $s-$reduced residue system (mod $n$).

Some of the important properties of the Klee's function $\Phi_s$ are listed in \cite[Section 2]{chandran2022menon}. Most notably, it has a product formula
   \begin{align}
    \Phi_s(n)=n\prod\limits_{\substack{p^s\mid n\\
  p \text{ prime}}}\left(1-\frac{1}{p^s}\right).
     \end{align}
 In \cite[Theorem 3]{cohen1956some}, Cohen showed that $\phi_s(n)=n^s \prod\limits_{\substack{p^s\mid n\\
  p \text{ prime}}}  (1-\frac{1}{p^s})$. So  $\phi_s(n) = \Phi_s(n^s)$. Further $J_s$ and $\phi_s$  are related by the identity $J_s(n)=\phi_s(n)$  \cite[Theorem 5]{cohen1956some}.   Note also that $\Phi_1 = J_1=\phi_1 = \phi$.

 The Cohen-Ramanujan sum $ c_r^{(s)}(n)$ definition of which  is given in the introduction above has the following properties.
 
 \begin{enumerate}[{CRS}1) ]
 \item $ c_r^{(s)}(n)=\sum\limits_{\substack{d|r\\d^s|n}}\mu(r/d)d^s$  where $\mu$ is the M\"{o}bius function \cite{cohen1949extension}.
\item $ c_r^{(s)}(n)=\frac{\phi_s(r)\mu(d)}{\phi_s(d)}=\frac{\Phi_s(r^s)\mu(d)}{\Phi_s(d^s)}$ where $d^s=\frac{r^s}{(n,r^s)_s}$ \cite{cohen1956extension}.
 \item $c_r^s(n)=c_r^s(-n)$.
 \item $c_r^s(1)=\mu(r)$.
 \end{enumerate}

   \section{Main Results}

We will derive the Menon-type identity (\ref{this-paper-idn}) in this section.  To begin with, we prove three lemmas.
   
 \begin{lemm}\label{lem1}
 For positive integers $n,s$, we have
 $\frac{n^s}{\Phi_s(n^s)}=\sum\limits_{\substack{d|n}}\frac{(\mu(d))^2}{\Phi_s(d^s)}$.
 \end{lemm}
 \begin{proof}
  Let $n=p_1^{a_1}p_2^{a_2}\cdots p_k ^{a_k}$. Then
  \begin{align*}
\frac{n^s}{\Phi_s(n^s)}&=\frac{1}{\prod\limits_{\substack{p^s\mid n^s\\
  p \text{ prime}}}\left(1-\frac{1}{p^s}\right)}\\
  &= \prod\limits_{\substack{p^s\mid n^s\\
  p \text{ prime}}} (1+\frac{1}{p^s-1})\\
  &=  \left(1+\frac{1}{p_1^s-1}\right)\left(1+\frac{1}{p_2^s-1}\right)\cdots \left(1+\frac{1}{p_k^s-1}\right)\\
     &= \prod\limits_{\substack{i=1}}^{k}\left( 1+\frac{(\mu(p_i))^2}{\Phi_s(p_i^s)}\right)\\
&= \sum\limits_{\substack{d|n}}\frac{(\mu(d))^2}{\Phi_s(d^s)}.
\end{align*}
  
 \end{proof}

 An element in $\mathbb{N}^2$ will be denoted by $(a;b)$ to avoid any confusion with the notation of gcd of $a$ and $b$. The next result generalizes  \cite[Proposition 1]{toth2011weighted}.
 
  \begin{lemm}\label{lem2}
 Let $\psi: \N^2\rightarrow \C$ be an arbitrary function and $k,n,s$ positive integers. Then we have$$\sum\limits_{\substack{k=1}}^{n^s}\psi(k;n^s)(k,n^s)_s= \sum\limits_{\substack{d^s\mid n^s}}\Phi_s(n^s)\sum\limits_{\substack{j=1}}^{\frac{n^s}{d^s}}\psi(d^sj;n^s).$$

 \end{lemm}
 \begin{proof}
 
 We have  $n^s =  \sum\limits_{d|n}J_s(d)$ \cite[Section V.3]{sivaramakrishnan1988classical} and $J_s(n) = \Phi_s(n^s)$. Hence $n^s =  \sum\limits_{d|n}\Phi_s(d^s)=\sum\limits_{d^s|n^s}\Phi_s(d^s)$. So
 
 \begin{align*}
 \sum\limits_{\substack{k=1}}^{n^s}\psi(k;n^s)(k,n^s)_s&=
 \sum\limits_{\substack{k=1}}^{n^s}\psi(k;n^s)\sum\limits_{d^s|(k,n^s)_s}\Phi_s(d^s)\\
 &=\sum\limits_{\substack{k=1}}^{n^s}\psi(k;n^s)\sum\limits_{\substack{d^s\mid k\\d^s\mid n^s}}\Phi_s(d^s).
 \end{align*}
 Since $d^s\mid k$ and $1\leq k\leq n^s$, we may write $k=d^sj $ where $1\leq j\leq \frac{n^s}{d^s}$ and hence we get the first part of the lemma $ \sum\limits_{\substack{k=1}}^{n^s}\psi(k;n^s)(k,n^s)_s=
 \sum\limits_{\substack{d^s\mid n^s}}\Phi_s(d^s) 
 \sum\limits_{\substack{j=1}}^{\frac{n^s}{d^s}}\psi(d^sj;n^s)$.

 \end{proof}
 
  \begin{lemm}\label{lem3}For $n,s,j$ positive integers, we have that
$$\sum\limits_{\substack{b=1}}^{n^s}(b,n^s)_se^{\frac{2\pi ibj}{n^s}}=\sum\limits_{\substack{d^s\mid (j,n^s)_s}}d^s\Phi_s(\frac{n^s}{d^s}).$$
 \end{lemm}
 
 \begin{proof}
 Define $\psi(b ;n)=e^{\frac{2\pi ibj}{n}}$ so that
 \begin{align*}
 \sum\limits_{\substack{b=1}}^{n^s}(b,n^s)_se^{\frac{2\pi ibj}{n^s}}&
 =\sum\limits_{\substack{b=1}}^{n^s}(b,n^s)_s\psi(b;n^s)\\&
  =\sum\limits_{\substack{d^s\mid n^s}}\Phi_s({d^s})\sum\limits_{\substack{k=1}}^{\frac{n^s}{d^s}}\psi({k},{\frac{n^s}{d^s}})\\&
 =\sum\limits_{\substack{d^s\mid n^s}}\Phi_s({d^s})\sum\limits_{\substack{k=1}}^{\frac{n^s}{d^s}}e^{\frac{2\pi i kj}{\frac{n^s}{d^s}}}.
 \end{align*}
 Note that $\frac{1}{d^s}\sum\limits_{\substack{k=1}}^{d^s}e^{\frac{2\pi i kn}{d^s}}=\begin{cases}1 \text{ if } d^s \mid n\\
 0 \text{ otherwise } .
 \end{cases}$

 So
 \begin{align*}
 \sum\limits_{\substack{b=1}}^{n^s}(b,n^s)_se^{\frac{2\pi ibj}{n^s}}&= \sum\limits_{\substack{d^s\mid n^s\\ \frac{n^s}{d^s}\mid j}}\Phi_s({d^s}){\frac{n^s}{d^s}}\\&
 =\sum\limits_{\substack{d^s\mid n^s\\ {d^s}\mid j}}\Phi_s({\frac{n^s}{d^s}}){{d^s}}\\&
 =\sum\limits_{\substack{ {d^s}\mid (j,n^s)_s}}\Phi_s({\frac{n^s}{d^s}}){{d^s}}. 
 \end{align*}

 \end{proof}

 Now we are ready to derive identity (\ref{this-paper-idn}) which we restate below.
 \begin{theo}\label{theo1}
 For $n,s$ positive integers, we have that
  $$\sum\limits_{\substack{a=1\\(a,n^s)_s=1}}^{n^s} (a-1,n^s)_s=\Phi_s(n^s)\tau_s(n^s).$$
 \end{theo}
 \begin{proof}
 
 \begin{align*}
 \sum\limits_{\substack{a=1\\(a,n^s)_s=1}}^{n^s} (a-1,n^s)_s&=  \sum\limits_{\substack{a,b=1\\(a,n^s)_s=1\\a-1\equiv b(\text{mod }n^s)}}^{n^s} (b,n^s)_s
  \\&= \sum\limits_{\substack{a,b=1\\(a,n^s)_s=1}}^{n^s}  \sum\limits_{\substack{a-1\equiv b(\text{mod }n^s)}} (b,n^s)_s
 \\&=  \sum\limits_{\substack{a,b=1\\(a,n^s)_s=1}}^{n^s} (b,n^s)_s \frac{1}{n^s} \sum\limits_{\substack{j=1}}^{n^s}e^{\frac{2\pi ij(1-a+b)}{n^s}}.
\end{align*} 
 To arrive at the last line above, we used the fact that $\frac{1}{d^s}\sum\limits_{\substack{k=1}}^{d^s}e^{\frac{2\pi i mn}{d^s}}=1$ if $d^s|n$.
Continuing with the computation, we get
 \begin{align*}
\sum\limits_{\substack{a=1\\(a.n^s)_s=1}}^{n^s} (a-1,n^s)_s &=  \frac{1}{n^s} \sum\limits_{\substack{j(\text{mod }n^s)}} e^{\frac{2\pi i j}{n^s}}\sum\limits_{\substack{b(\text{mod }n^s)}}(b,n^s)_s e^{\frac{2\pi ibj}{n^s}}\sum\limits_{\substack{a=1\\(a,n^s)_s=1}}^{n^s}e^{\frac{2\pi ia(-j)}{n^s}}\\
 &= \frac{1}{n^s} \sum\limits_{\substack{j(\text{mod }n^s)}} e^{\frac{2\pi i j}{n^s}}\sum\limits_{\substack{b(\text{mod }n^s)}}(b,n^s)_s e^{\frac{2\pi ibj}{n^s}}c_n^s(j).
 \end{align*}

Now let $(j,n^s)_s=\delta^s\Rightarrow \delta^s\mid n^s, \delta^s\mid j
 \Rightarrow \delta^s\mid n^s, j=t\delta^s$ for some integer $t$.
 Also $(\frac{j}{\delta^s}, \frac{n^s}{\delta^s} )_s=1$. That is $(t, \frac{n^s}{\delta^s} )_s=1$.
 
 Since $\frac{n^s}{(j,n^s)_s}=\frac{n^s}{\delta^s}=(\frac{n}{\delta})^s$, by property CRS2 of the Cohen-Ramanujan sum given in section 2 and Lemma (\ref{lem3}), we get

 \begin{align*}
 \sum\limits_{\substack{a=1\\(a.n^s)_s=1}}^{n^s} &(a-1,n^s)_s\\&=\frac{1}{n^s} \sum\limits_{\substack{t=1\\(t, \frac{n^s}{\delta^s} )_s=1\\\delta^s\mid n^s}}^{\frac{n^s}{\delta^s}} e^{\frac{2\pi i t\delta^s}{n^s}}\sum\limits_{\substack{d^s\mid \delta^s}}d^s\Phi_s(\frac{n^s}{d^s})\frac{\Phi_s(n^s)\mu(\frac{n}{\delta})}{\Phi_s(\frac{n^s}{\delta^s})}
 \\&=\frac{\Phi_s(n^s)}{n^s} \sum\limits_{\substack{\delta^s\mid n^s}}\frac{\mu(\frac{n}{\delta})}{\Phi_s(\frac{n^s}{\delta^s})}\sum\limits_{\substack{d^s\mid \delta^s}}d^s\Phi_s(\frac{n^s}{d^s})\sum\limits_{\substack{t=1\\(t, \frac{n^s}{\delta^s} )_s=1}}^{\frac{n^s}{\delta^s}} e^{\frac{2\pi i t\delta^s}{n^s}}
 \\&=\frac{\Phi_s(n^s)}{n^s} \sum\limits_{\substack{\delta^s\mid n^s}}\frac{\mu(\frac{n}{\delta})}{\Phi_s(\frac{n^s}{\delta^s})}\sum\limits_{\substack{d^s\mid \delta^s}}d^s\Phi_s(\frac{n^s}{d^s})c_{\frac{n}{\delta}}^s(1)
  \\&=\frac{\Phi_s(n^s)}{n^s} \sum\limits_{\substack{\delta^s\mid n^s}}\frac{(\mu(\frac{n}{\delta}))^2}{\Phi_s(\frac{n^s}{\delta^s})}\sum\limits_{\substack{d^s\mid \delta^s}}d^s\Phi_s(\frac{n^s}{d^s})\text{ since } \mu(\frac{n}{\delta}) = c_{\frac{n}{\delta}}^s(1) \text{ (by CRS4)}.
 \end{align*}
 
 Now $d^s\mid \delta^s, \delta^s\mid n^s\Rightarrow n^s=\delta^s h^s=d^sk^sh^s$ so that

\begin{align*}
 \sum\limits_{\substack{a=1\\(a,n^s)_s=1}}^{n^s} (a-1,n^s)_s&=\frac{\Phi_s(n^s)}{n^s} \sum\limits_{\substack{n^s=d^sk^sh^s}} d^s\Phi_s(k^sh^s)\frac{(\mu(h))^2}{\Phi_s(h^s)}\\
&=\frac{\Phi_s(n^s)}{n^s} \sum\limits_{\substack{n^s=d^sm^s}} d^s\Phi_s(m^s)\sum\limits_{\substack{m=kh}}\frac{(\mu(h))^2}{\Phi_s(h^s)}
   \\&=\frac{\Phi_s(n^s)}{n^s} \sum\limits_{\substack{n^s=d^sm^s}} d^s\Phi_s(m^s)\frac{m^s}{\Phi_s(m^s)} \hspace{10pt}\text{(by Lemma \ref{lem1})}
   \\& =\frac{\Phi_s(n^s)}{n^s} \sum\limits_{\substack{n^s=d^sm^s}} d^s{m^s}
     \\&=\frac{\Phi_s(n^s)}{n^s} \sum\limits_{\substack{d^s|n^s}}n^s
           \\&={\Phi_s(n^s)} \tau_s(n^s)
\end{align*}
which is what we claimed.

\end{proof}

\section{Acknowledgements}

The first author thanks the University Grants Commission of India for providing financial support for carrying out research work through their Senior Research Fellowship (SRF) scheme.


\begin{thebibliography}{10}

\bibitem{chandran2022menon}
Arya Chandran, Neha~Elizabeth Thomas, and K~Vishnu Namboothiri.
\newblock A Menon-type identity using Klee’s function.
\newblock {\em Czechoslovak Mathematical Journal}, 72(1):165--176, 2022.

\bibitem{cohen1949extension}
Eckford Cohen.
\newblock An extension of {Ramanujan's} sum.
\newblock {\em Duke Mathematical Journal}, 16(85-90):2, 1949.

\bibitem{cohen1956extension}
Eckford Cohen.
\newblock An extension of Ramanujan’s sum. iii. Connections with totient
  functions.
\newblock 1956.

\bibitem{cohen1956some}
Eckford Cohen.
\newblock Some totient functions.
\newblock {\em Duke Mathematical Journal}, 23(4):515--522, 1956.

\bibitem{jordan1870traite}
Camille Jordan.
\newblock {\em Traite des substitutions et des equations algebriques par m.
  Camille Jordan}.
\newblock Gauthier-Villars, 1870.

\bibitem{li2017menon}
Yan Li and Daeyeoul Kim.
\newblock A Menon-type identity with many tuples of group of units in
  residually finite Dedekind domains.
\newblock {\em Journal of Number Theory}, 175:42--50, 2017.

\bibitem{menon1965sum}
P~Kesava Menon.
\newblock On the sum $\sum (a-1, n),[(a, n)= 1]$.
\newblock {\em The Journal of the Indian Mathematical Society}, 29(3):155--163,
  1965.

\bibitem{ramanujan1918certain}
Srinivasa Ramanujan.
\newblock On certain trigonometrical sums and their applications in the theory
  of numbers.
\newblock {\em Trans. Cambridge Philos. Soc}, 22(13):259--276, 1918.

\bibitem{sivaramakrishnan1988classical}
R~Sivaramakrishnan.
\newblock {\em Classical theory of arithmetic functions}, volume 126.
\newblock CRC Press, 1988.

\bibitem{sury2009some}
Balasubramanian Sury.
\newblock Some number-theoretic identities from group actions.
\newblock {\em Rendiconti del circolo matematico di Palermo}, 58(1):99--108,
  2009.

\bibitem{toth2018menon}
L{\'a}szl{\'o} T{\'o}th.
\newblock Menon-type identities concerning Dirichlet characters.
\newblock {\em International Journal of Number Theory}, 14(04):1047--1054,
  2018.

\bibitem{toth2021proofs}
L{\'a}szl{\'o} T{\'o}th.
\newblock Proofs, generalizations and analogs of Menon's identity: a survey.
\newblock {\em arXiv preprint arXiv:2110.07271}, 2021.

\bibitem{toth2011weighted}
L{\'a}szl{\'o} T{\'o}th and Gregor Mendel-Stra{\ss}e.
\newblock Weighted gcd-sum functions.
\newblock {\em Journal of Integer Sequences}, 14(2):3, 2011.

\bibitem{zhao2017another}
Xiao-Peng Zhao and Zhen-Fu Cao.
\newblock Another generalization of Menon’s identity.
\newblock {\em International Journal of Number Theory}, 13(9):2373--2379, 2017.

\end{thebibliography}
\end{document}